\documentclass[10pt,leqno,twoside]{amsart}
\numberwithin{equation}{section}
\usepackage{latexsym}
\usepackage{amsmath}
\usepackage{amssymb}
\usepackage{mathrsfs}
\usepackage{graphicx}
\usepackage{ifthen}
\usepackage[T1]{fontenc}
\usepackage[latin1]{inputenc}

\newtheorem{theorem}{Theorem}[section]
\newtheorem{proposition}[theorem]{Proposition}

\newtheorem{lemma}[theorem]{Lemma}

\theoremstyle{definition}

\theoremstyle{definition} %%{remark}
\newtheorem{remark}[theorem]{Remark}

\newcommand{\cA}{\mathcal A}
\newcommand{\cB}{\mathcal B}
\newcommand{\cD}{\mathcal D}
\newcommand{\cE}{\mathcal E}

\newcommand{\R}{\mathbb R}

\newcommand{\N}{\mathbb N}

\newcommand{\sE}{{\sf E}}

\newcommand{\IP}{\mathbb P}

\DeclareMathSymbol{\complement}{\mathord}{AMSa}{"7B}

\def\vv<#1>{\langle #1\rangle}
\def\Vv<#1>{\bigl\langle #1\bigr\rangle}

\begin{document}

% TOPMATTER

\title[Dynamics of Nematic Liquid Crystal Flows]
{Dynamics of Nematic Liquid Crystal Flows:\\ the Quasilinear Approach}

\author{M. Hieber}
\address{Technische Universit\"at Darmstadt\\
        Fachbereich Mathematik\\
        Schlossgarten-Strasse 7\\
        D-64289 Darmstadt, Germany}
\email{hieber@mathematik.tu-darmstadt.de}
\author{M. Nesensohn}
\address{Technische Universit\"at Darmstadt\\
        Fachbereich Mathematik\\
        Schlossgarten-Strasse 7\\
        D-64289 Darmstadt, Germany}
\email{nesensohn@mathematik.tu-darmstadt.de}
\author{J. Pr\"uss}
\address{Martin-Luther-Universit\"at Halle-Witten\-berg\\
         Institut f\"ur Mathematik \\
         Theodor-Lieser-Strasse 5\\
         D-06120 Halle, Germany}
\email{jan.pruess@mathematik.uni-halle.de}
\author{K. Schade}
\address{Technische Universit\"at Darmstadt\\
        Fachbereich Mathematik\\
        Schlossgarten-Strasse 7\\
        D-64289 Darmstadt, Germany}
\email{schade@mathematik.tu-darmstadt.de}


\thanks{The second and fourth author are supported by the
DFG International Research Training Group 1529 on Mathematical Fluid Dynamics at
TU Darmstadt}

\begin{abstract}
%The Leslie-Erickson model for the flow of nematic liquid crystals is proved to be well-posed in the $L_p$-setting. The energy of the system is shown to be a strict Ljapunov-functional and the equilibria of the system are identified, they form a manifold of dimension $n-1$. It is proved that they are stable in the natural state space and that any solution which stays bounded in the state space convergence to an equilibrium.
%These results are obtained by using modern parabolic theory including weighted $L_p$ spaces on an equivalent reformulation of the LCD problem as a quasilinear equation.
  Consider the (simplified) Leslie-Erickson model for the flow of nematic liquid crystals in a bounded domain $\Omega \subset \mathbb{R}^n$ for $n > 1$. This article develops a complete dynamic theory for these equations, analyzing the system as a quasilinear parabolic evolution equation in an $L_p-L_q$-setting. First, the existence of a unique local strong solution is proved. This solution extends to a global strong solution, provided the initial data are close to an equilibrium or the solution is eventually bounded in the natural norm of the underlying state space. In this case the solution converges exponentially to an equilibrium. Moreover, the solution is shown to be real analytic, jointly in time and space.
\end{abstract}

\maketitle

\section{Introduction}

We consider the following system modeling the flow of nematic
liquid crystal materials in a bounded domain $\Omega \subset \mathbb{R}^n$
\begin{align}\label{eq:LCD}
\left\{
\begin{array}{rllll}
  \partial_tu + (u \cdot \nabla) u - \nu\Delta u + \nabla\pi &\!=\!&
-\lambda {\rm div}([\nabla d]^{\sf T}\nabla d) && \text{in } (0,T) \times \Omega,  \\
  \partial_td  + (u \cdot \nabla)d &\!=\!& \gamma(\Delta d +  |\nabla d|^2d) && \text{in } (0,T) \times \Omega, \\
  {\rm div}\, u &\!=\!& 0 && \text{in } (0,T) \times \Omega,\\
  (u,\partial_\nu d) &\!=\!& (0,0)&& \text{on } (0,T) \times \partial\Omega, \\
  (u,d)_{\vert t=0} &\!=\!& (u_0,d_0) && \text{in }   \Omega.
 \end{array}\right.
\end{align}
Here, the function $u: (0,\infty) \times \Omega \to \mathbb{R}^n$ describes the
velocity field, $\pi: (0,\infty) \times \Omega \to \mathbb{R}$ the pressure, and
$d:(0,\infty) \times \Omega \to \R^n$ represents the macroscopic molecular
orientation of the liquid crystal.  Due to the  physical interpretation  of $d$ it is natural
to impose the condition
\begin{equation}\label{deq1}
  |d|= 1  \qquad \text{in } (0,T) \times \Omega.
\end{equation}
We will show in the following that this condition is indeed preserved by
the above  system; see Proposition \ref{lengthdpreserved}
below for details.

The constants $\nu>0,\lambda>0$, and
$\gamma>0$ represent viscosity, the competition between kinetic energy and
potential energy and the microscopic elastic relaxation time for the molecular
orientation field, respectively. For simplicity, we set $\nu=\lambda=\gamma =1$,
which does not change our analysis.

The continuum theory of liquid crystals was developed by Ericksen and Leslie during the
1950's and 1960's in \cite{ericksen62,leslie68}. The Ericksen-Leslie theory is widely
used as a model for the flow of liquid crystals, see for example the survey articles by
Leslie in \cite{kinderlehrerericksen} and also
\cite{chandrasekhar, deGennes,hardtkinderlehrerlin86, lin89}.

The set of equations \eqref{eq:LCD} was considered  first in \cite{linliu95}, however for
the situation where in the second equation of \eqref{eq:LCD} the term $|\nabla d|^2d$
is replaced by $f(d) = \nabla F(d)$, i.e.
 \begin{align*}
 d_t + (u \cdot \nabla) d = \gamma(\Delta d - f(d)),
\end{align*}
where $F:\R^3 \to \R$ is a smooth, bounded function. Note that in this situation, the
condition \eqref{deq1} {\it cannot} be preserved in general. Thus, this condition was replaced
in  \cite{lin89} and \cite{linliu95} by the Ginzburg-Landau energy functional, i.e.
$f$ is assumed to satisfy $f(d) = \nabla F(d) = \nabla \frac{1}{\varepsilon^2}(|d|^2-1)^2$.
In 1995, Lin and Liu \cite{linliu95} proved  the existence of global weak solutions
to \eqref{eq:LCD} in dimension $2$ or $3$ under the assumptions that
$u_0 \in L_2(\Omega), d_0 \in H^1(\Omega)$, and $d_0 \in H^{3/2}(\partial\Omega)$.
Existence and uniqueness of global classical solutions was also obtained by them in dimension
$2$ provided $u_0 \in H^1(\Omega), d_0 \in H^2(\Omega)$, and provided the viscosity $\nu$
is large in dimension $3$. For regularity results of weak solutions
in the spirit of Caffarelli-Kohn-Nirenberg we refer to
\cite{linliu96}.

Hu and Wang \cite{huwang10} considered in 2010 the case of $f(d)=0$ and proved
existence and uniqueness of a global strong solution for small initial data in this case.
They proved moreover that whenever a strong solutions exist, all global weak solutions
as constructed in \cite{linliu95} must be equal to this strong solution.
The idea of their  approach was to consider the above system  \eqref{eq:LCD} as a semilinear
equation with a forcing  term $\lambda {\rm div}([\nabla d]^{\sf T}\nabla d)$ on the
right-hand side.

The full system  \eqref{eq:LCD} with $f(d) = |\nabla{d}|^2 d$ was revisited by
Lin, Lin, and Wang in 2010. They proved in \cite{linlinwang10} interior and boundary
regularity theorems under smallness condition in dimension $2$ and established
the  existence of global weak solutions on bounded smooth domains $\Omega \subset \R^2$
that are smooth away from a finite set. Furthermore, Wang proved in~\cite{Wan11} global
well-posedness for this system for initial data being small in $BMO^{-1}\times BMO$
in the case of a whole space, i.e. $\Omega =\R^n$, by combining techniques of
Koch and Tataru with methods from harmonic maps to certain Riemannian manifolds. 

For results on the compressible case we refer to \cite{HWW12, Ma13}. Here, local existence of
strong solutions is proved. The latter turn out to be even local classical solution.

Summarizing, we observe that in particular results for local as well as global, strong solutions in the three
dimensional setting for the full system \eqref{eq:LCD}, obeying
also the condition \eqref{deq1}, do not seem to exist so far.

Recently, Li and Wang claimed in \cite{li-wang} such  a result. More precisely, they claimed
the existence and uniqueness of a strong solutions to \eqref{eq:LCD} in bounded,
smooth domains (however, not satisfying \eqref{deq1}). Their idea was to rewrite \eqref{eq:LCD}
as a {\it semilinear} equation for the Stokes equation coupled to the heat equation with a right hand side of the form
\begin{align*}
\widetilde{F}(u,d):= \left(-(u\cdot\nabla)u -
{\rm div}([\nabla d]^{\sf T}\nabla d), -(u\cdot\nabla)d+|\nabla{d}|^2d\right).
\end{align*}
Unfortunately, their approach and their main result \cite[Theorem2.1]{li-wang} relies on an
{incorrect} regularity property  for the solution of the heat equation
\cite[Theorem 3.1]{li-wang}. This result would imply further regularity properties for $d$ and
hence for $\widetilde{F}(u,d)$, which however are not true. Note that the (incorrect)
assertion of \cite[Theorem 3.1]{li-wang} is crucial for their approach. Thus, the
theory for local as well as for global strong solutions to \eqref{eq:LCD},  also satisfying
\eqref{deq1}, needs clarification.

It is the aim of this paper to present a complete theory for global strong
solutions to \eqref{eq:LCD} satisfying \eqref{deq1} as well as for their dynamical
behaviour in the $n$-dimensional setting, where $n > 1$.

Our main idea is to consider \eqref{eq:LCD} not as a semilinear equation as done
in all of the previous approaches but as a  {\it quasilinear} evolution
equation. We thus  incorporate the term ${\rm div}([\nabla d]^{\sf T}\nabla d)$
%\label{divpart}
into the quasilinear  operator $A$ given by
\begin{align*}
A(d)=\left[\begin{array}{cc} \cA & \IP \cB(d)\\
                                 0 & \cD\end{array}\right],
\end{align*}
where $\cA$ denotes the Stokes operator, $\cD$ the Neumann Laplacian, and $\cB$
is given by
\begin{align*}
[\cB(d) h]_i := \partial_i d_l \Delta h_l +\partial_k d_l \partial_k\partial_i h_l,
\end{align*}
for which we employ the sum convention. Note that $\mathcal{B}(d)d = {\rm div}([\nabla d]^{\sf T}\nabla d)$.

We then develop a complete dynamic theory for~\eqref{eq:LCD}--\eqref{deq1}.
In fact, first by local existence theory for abstract quasilinear parabolic problems,
we prove the existence and uniqueness of a strong solution to \eqref{eq:LCD}--\eqref{deq1}
on a maximal time interval. Thus,  \eqref{eq:LCD}--\eqref{deq1} give rise to a
local semi-flow in the natural state space.

Furthermore, the equilibria $\cE$ of~\eqref{eq:LCD}--\eqref{deq1} are determined to be
 \begin{align*}
        \cE=\{(0,d_*): \; d_*\in \R^n,\, |d_*|=1\},
 \end{align*}
and the well-known energy functional
$$
{\sf E}=\frac{1}{2} \int_\Omega [|u|^2 + |\nabla d|^2] dx
$$
for~\eqref{eq:LCD}--\eqref{deq1} is shown to be a strict Ljapunov-functional.
In addition, the equilibria are shown to be normally stable, i.e.\ for an initial value close
to $\cE$, the solution of \eqref{eq:LCD}--\eqref{deq1} exists globally and
the solution converges exponentially to an equilibrium. More generally, a solution,
eventually bounded on its maximal interval of existence, exists
globally and converges to an equilibrium exponentially fast.

Due to the polynomial character of the nonlinearities, we can even show that the solution 
of \eqref{eq:LCD}--\eqref{deq1} is real analytic, jointly in time and space.

Our approach is based on the theory of quasilinear parabolic problems and
relies in particular  on the maximal $L_p$-regularity property for the heat and the
Stokes equation. In particular, we refer here to \cite{Ama95, Ama00, DHP03, clementli, KPW09, Pru03, PrSi04, psz09}.

The plan for this paper is as follows.
We begin by collecting general results from the theory of quasilinear
parabolic evolution equations. Then, in Section \ref{abstractreformulation}
we introduce our formulation of \eqref{eq:LCD}. Section~\ref{sec:local} deals with
local well-posedness and regularity of solutions to \eqref{eq:LCD}--\eqref{deq1};
in particular we see that the solution is real analytic.
The generalized principle of linearized stability yields the stability of
equilibria and convergence of solutions is proved in Section \ref{stabilityofequilibria}. Moreover,
by means of the associated energy functional, we prove convergence of a solution to an equilibrium, whenever the solution is eventually bounded
in the natural state space.

\section{Quasilinear Evolution Equations} \label{toolbox}

Let $X_0$ and $X_1$ be Banach spaces such that $X_1\overset{d}{\hookrightarrow} X_0$,
i.e.\@ $X_1$ is continuously and densely embedded in $X_0$. Let $J=[0,a]$ for an $a>0$.
By a {\em quasilinear autonomous parabolic evolution equation} we understand an equation
of the form
\begin{align}\label{eq:qlproblem}\tag{QL}
 \dot{z}(t) + A(z(t))z(t) = F(z(t)),\quad t \in J, \quad z(0)=z_0,
\end{align}
where $A$ is a mapping from a real interpolation space $X_{\gamma,\mu}$ with suitable weights
between $X_0$ and $X_1$ into ${\mathcal L}(X_0,X_1)$. Our approach relies  on the maximal
$L_p$-regularity of $A(v)$ for $v \in X_{\gamma,\mu}$. For details we refer e.g. to
\cite{DHP03}.

The equation \eqref{eq:qlproblem} is investigated in spaces of the form $L_p(J;X_0)$
with temporal weights. More
precisely, for $p \in (1, \infty)$ and $\mu\in(1/p,1]$, the spaces $L_{p,\mu}$ and $H^1_{p,\mu}$ are defined by
\begin{align*}
  %z\in L_{p,\mu}(J;X_1)&\quad  :\Leftrightarrow \quad t^{1-\mu}z\in L_p(J;X_1),
  %\\ z\in H^1_{p,\mu}(J;X_0) &\quad :\Leftrightarrow
%\quad t^{1-\mu} z\in H^1_{p,\mu}(J;X_0).
    L_{p,\mu}(J;X_1) &:= \{ z\colon J \to X_1\colon t^{1-\mu}z\in L_p(J;X_1) \},
    \\H^1_{p,\mu}(J;X_0) &:= \{ z\in L_{p,\mu}(J;X_0) \cap W^1_1(J;X_0) \colon \dot{z} \in L_{p,\mu}(J;X_0) \}.
\end{align*}
It is clear, that
\begin{align*}
 L_p(J;X) \hookrightarrow L_{p,\mu}(J;X) \quad \text{and} \quad L_p([0,a];X)  \hookrightarrow L_{p,\mu}([\tau,a];X),
\end{align*}
for all Banach spaces $X$ and $\tau \in (0,a)$. 
It has been shown in \cite[Theorem 2.4]{PrSi04} that $L_p$-maximal regularity implies
also $L_{p,\mu}$-maximal regularity, provided $p\in (1, \infty)$ and $\mu \in (1/p, 1]$.
The trace space of the maximal regularity class containing temporal weights,
\begin{align*}z\in H^1_{p,\mu}(J;X_0)\cap L_{p,\mu}(J;X_1)\end{align*}
has been characterized in \cite[Theorem 2.4]{PrSi04}  as
\begin{align*}
  X_{\gamma,\mu}=(X_0,X_1)_{\mu-1/p,p},
\end{align*}
provided $p\in (1, \infty)$ and $\mu \in (1/p, 1]$; see also \cite[Theorem 4.2]{MS12}.

We now impose precise  assumptions on $A$ and $F$.
\begin{itemize}
\item[(A)] $A\in C^\omega(X_{\gamma,\mu}; \mathcal{L}(X_0, X_1))$, and
$A(v)$ has maximal $L_p$-regularity for each $v \in X_{\gamma,\mu}$.
\item[(F)] $F \in C^\omega(X_{\gamma, \mu}; X_0)$.
\end{itemize}
Even under less restrictive Lipschitz type assumptions on $A$ and $F$, local in time
existence of~\eqref{eq:qlproblem} was shown by Cl\'{e}ment-Li~\cite{clementli} in the
case $\mu = 1$ and by K\"ohne-Pr\"uss-Wilke \cite[Theorem 2.1, Corollary 2.2]{KPW09} for
the case $\mu \in (1/p, 1]$.

\begin{proposition}%[Cl\'{e}ment-Li \cite{clementli},
%K\"ohne-Pr\"uss-Wilke \cite[Theorem 2.1]{KPW09}]
\label{clementli}
Let $1<p<\infty$, $\mu \in (1/p,1]$, $z_0 \in X_{\gamma, \mu}$, and suppose that
the assumptions (A) and (F) are satisfied. Then, there exists $a > 0$, such that
\eqref{eq:qlproblem} admits a unique solution $z$ on $J = [0,a]$ in the regularity class
\begin{align*}
    z \in H^1_{p, \mu}(J; X_0) \cap L_{p,\mu}(J; X_1) \hookrightarrow
C(J; X_{\gamma, \mu})\cap C( (0,a]; X_\gamma).
\end{align*}
The solution depends continuously on $z_0$, and can be extended to a maximal interval of
existence $J(z_0) = [0, t^+(z_0))$.
\end{proposition}

Parabolic problems allow for additional smoothing effects.
In this respect, a method due to Angenent \cite{Ang90} is well known. We only state here a
variant of it which is adapted to \eqref{eq:qlproblem};
see \cite[Theorem 5.1]{Pru03} for the case $\mu = 1$. By a slight adjustment
of its  proof to the situation of temporal weights, this result remains true also for
maximal regularity classes using this type of weights.

\begin{proposition} \label{qlregularity}
Let $1<p<\infty$, $\mu\in(1/p,1]$, $a>0$, and assume that (A) and (F) hold.
Let $z \in H^1_{p, \mu}(J; X_0) \cap L_{p,\mu}(J; X_1)$ be a solution of~\eqref{eq:qlproblem}
on $J = [0,a]$. Then
\begin{align*}
   t^k [\frac{d}{dt}]^k z \in  H^1_{p,\mu}(J;X_0)\cap L_{p,\mu}(J;X_1), \quad k \in \mathbb{N}.
\end{align*}
Furthermore, $z$ is real analytic with values in $X_1$ on $(0,a)$.
\end{proposition}

We denote the set of equilibria of~\eqref{eq:qlproblem} by
\begin{align*}
  \mathcal{E} = \{ z_* \in X_1\colon A(z_*)z_* = F(z_*) \}.
\end{align*}
The following result on global existence and stability of was proved in
\cite[Theorem 2.1]{psz09} assuming only that $A$ and $F$ are of class $C^1$.
%They considered the classical case $\mu = 1$.

\begin{proposition}\label{thm:global1}
Let $1<p<\infty$ and assume that assumptions (A) and (F) hold.  Let $A_0$ be the
linearization of~\eqref{eq:qlproblem}, i.e. let
\begin{align*}
        A_0 w= A(z_*)w + (A'(u_*)w)u_* - F'(u_*)w, \quad w\in X_1.
\end{align*}
Suppose that $u_* \in \mathcal{E}$ is normally stable equilibrium, i.e.
\begin{itemize}
\item[(i)] near $u_*$ the set of equilibria $\mathcal{E} \subset X_1$ is a
$C^1$-manifold in $X_1$ of dimension $m\in \mathbb{N}_0$,
\item[(ii)] the tangent space of $\mathcal{E}$ at $u_*$ is given by $N(A_0)$,
\item[(iii)] 0 is semi-simple eigenvalue of $A_0$, i.e. $N(A_0) \oplus R(A_0) = X_0$,
\item[(iv)] $\sigma(A_0)\setminus \{0\} \subset \mathbb{C}_+ =
\{x \in \mathbb{C}\colon \operatorname{Re} x > 0\}$.
\end{itemize}
Then $u_*$ is stable in $X_\gamma$.  Further, there exists  a number $\rho > 0$ such
that the unique solution $z$ of~\eqref{eq:qlproblem} with initial value
$z \in B_{X_\gamma}(0,\rho)$ exists on $\mathbb{R}_+$ and converges at an
exponential rate to some $u_\infty \in \mathcal{E}$ in $X_\gamma$ as $t\to\infty$.
\end{proposition}

We finish the section with another result on global existence result for~\eqref{eq:qlproblem}; see
\cite[Theorem 3.1]{KPW09}.

\begin{proposition}\label{thm:global2}
Let $1<p<\infty$, $\mu \in (1/p,1]$, $z_0 \in X_{\gamma, \mu}$ and let
$J = [0,a]$ or $J = \mathbb{R}_+$. Suppose that assumptions (A) and (F) are
satisfied and that the embedding $X_{\gamma, \mu} \overset{c}{\hookrightarrow} X_\gamma$
is compact. Assume furthermore that the solution $z$ of~\eqref{eq:qlproblem} is eventually bounded in $ X_{\gamma}$ on its maximal interval of existence, i.e. that $z$ satisfies
  \begin{align*}
    z \in BC([\tau, t^+(z_0)); X_{\gamma})
  \end{align*}
  for some $\tau \in (0,t^+(z_0))$. Then the solution $z$ exists globally and for each
$\delta >0$, the orbit $\{z(t)\}_{t \geq \delta}$ is relatively compact in $X_\gamma$. If in addition $z_0 \in X_\gamma$, then $\{z(t)\}_{t \geq 0}$ is relatively compact in $X_\gamma$.
\end{proposition}

\section{Nematic liquid crystals as quasilinear evolution equations}
\label{abstractreformulation}

We now reformulate \eqref{eq:LCD} equivalently as a quasilinear parabolic evolution
equation for the unknown $z=(u,d)$. To this end, for $1<q<\infty$ define the
Banach spaces $X_0$ by
\begin{align*}
X_0:= L_{q,\sigma}(\Omega)\times L_q(\Omega)^n,
\end{align*}
where $\Omega\subset \R^n$ is a bounded domain with boundary $\partial\Omega\in C^2$.
The subscript $\sigma$ in $L_{q,\sigma}(\Omega)$ means as usual the subspace of $L_q(\Omega)^n$
consisting of solenoidal vector fields.

The Neumann-Laplacian $\cD_q$ in  $L_q(\Omega)$ is  defined by $\cD_q=-\Delta$ with domain
\begin{align*}
  D(\cD_q):= \{d\in H^2_q(\Omega)^n:\, \partial_\nu d =0 \mbox{ on } \partial\Omega\}.
\end{align*}
It is well-known that $\cD_q$ has the property of $L_p$-maximal regularity;
see~\cite[Theorem 8.2]{DHP03}.

Let $\IP: L_q(\Omega)^n \to L_{q,\sigma}(\Omega)$ denote the Helmholtz projection.
We then define the Stokes Operator $\cA_q=-\IP \Delta$ in $L_{q,\sigma}(\Omega)$ with domain
\begin{align*}
  D(\cA_q) = \{ u\in H^2_q(\Omega)^n:\, {\rm div}\, u=0 \mbox{ in } \Omega,\, u=0 \mbox{ on }
\partial\Omega\} .
\end{align*}
It is also well-known that $\cA_q$ has the property of $L_p$-maximal regularity;
see e.g. \cite{Sol77A,Gig85,ShortSolonnikov}.

Next, we define the space $X_1$ by
\begin{align*}X_1:= D(\cA_q)\times D(\cD_q),\end{align*}
  equipped with its canonical norms. Then $X_1\overset{d}{\hookrightarrow} X_0$ densely.

The quasilinear part $A(z)$ of \eqref{eq:qlproblem} is  given by the tri-diagonal matrix
\begin{align*} A(z)=\left[\begin{array}{cc} \cA_q & \IP \cB_q(d)\\
                                 0 & \cD_q\end{array}\right],\end{align*}
where the operator $\cB_q$ is given by
\begin{align*} [\cB_q(d) h]_i := \partial_i d_l \Delta h_l +\partial_k d_l \partial_k\partial_i h_l,
\end{align*}
for which we employed the sum convention. Note that
\begin{align*}
 \cB_q(d)d =  {\rm div}([\nabla d]^{\sf T}\nabla d).
\end{align*}

Obviously, $\cB(d):X_1\to X_0$  is bounded for each $d\in C^1(\overline{\Omega})^n$ and
the map $d\mapsto \IP\cB_q(d)$ is polynomial, hence real analytic.
%Hence the Lipschitz condition \textbf{(A)} imposed on the operator $A$ for the
%application of Theorem \ref{clementli} are fulfilled.
By the tri-diagonal structure of $A(z)$ and by the regularity of $\cB$ one can
easily see that $A(z)$ also has the property of $L_p$-maximal regularity,
for each $z\in C^1(\overline{\Omega})^{2n}$. Indeed, for a fixed right-hand side
$(f_u,f_d) \in L_p(0,a;X_{\gamma,\mu})$ and initial values $(u_0,d_0) \in X_{\gamma,\mu}$,
we may use the maximal regularity of $\cD_q$ to obtain a solution $\tilde{d}$ of the heat equation with
Neumann boundary condition in the right maximal regularity class. By setting
\begin{align*}
  \widetilde{f_u} := f_u - \IP\cB_q(d)\tilde{d}                                                              \end{align*}
  as right-hand side for the Stokes equation, we obtain a solution $\tilde{u}$ in the
right maximal regularity class due to the fact that $\cB_q(d)$ is linear and bounded.

The right-hand side $F(z)$ of \eqref{eq:qlproblem} is defined by
\begin{align*}
F(z)=(-\IP u\cdot\nabla u, -u\cdot \nabla d+|\nabla d|^2d),
\end{align*}
which is also polynomial, hence a real analytic mapping from
$C^1(\overline{\Omega})^{2n}$ into $X_0$. 

Note that (A) and (F) hold, as soon as we have the embedding 
$$X_{\gamma,\mu}\hookrightarrow C^1(\overline{\Omega})^{2n}.$$
The space $X_\gamma$ is given by
\begin{align*}
X_\gamma:= (X_0,X_1)_{1-1/p,p}= D_{\cA_q}(1-1/p,p)\times
D_{\cD_q}(1-1/p,p);
\end{align*}
see \cite{Ama95, DHP07}. As explained in Section \ref{toolbox}, we consider $L_p$-spaces
with temporal weights. The trace space of the class
$$z\in H^1_{p,\mu}(J;X_0)\cap L_{p,\mu}(J;X_1)$$
now reads
$$
X_{\gamma,\mu}=(X_0,X_1)_{\mu-1/p,p}= D_{\cA_q}(\mu-1/p,p)\times D_{\cD_q}(\mu-1/p,p),
$$
provided  $p\in (1,\infty)$ and $\mu\in(1/p,1]$; see ~\cite[Theorem 4.12]{MS12}.

In order to obtain the embeddings $X_\gamma\hookrightarrow C^1(\overline{\Omega})^{2n}$ and
more generally $X_{\gamma,\mu}\hookrightarrow C^1(\overline{\Omega})^{2n}$ we impose on $p,q \in (1,\infty)$
now the conditions
\begin{equation}\label{pq}
\frac{2}{p} +\frac{n}{q}<1,\quad \frac{1}{2} + \frac{1}{p} +\frac{n}{2q}<\mu\leq 1.
\end{equation}
Standard Sobolev embedding theorems can then be applied.  

Further, we recall from \cite[Theorem 4.3.3]{Tri78} and \cite[Theorem 3.4]{Ama00}, respectively,
the following characterizations of the interpolation spaces involved,
$$
d\in D_{\cD_q}( \mu-1/p,p) \quad\Leftrightarrow \quad d\in B^{2\mu-2/p}_{qp}(\Omega)^n,\;
\partial_\nu d=0 \mbox{ on } \partial\Omega,
$$
and
$$
u \in D_{\cA_q}(\mu-1/p,p) \quad \Leftrightarrow \quad  u\in B^{2\mu-2/p}_{qp}(\Omega)^n
\cap L_{q,\sigma}(\Omega),\; u=0 \mbox{ on } \partial\Omega.
$$
Observe that both of these characterizations make sense,
since the condition \eqref{pq} guarantees the existence of the trace.

\section{Existence, uniqueness, and regularity of solutions}
\label{sec:local}

We start this section by applying Proposition \ref{clementli} to obtain the
following result on local well-posedness of \eqref{eq:LCD}.

\begin{theorem}\label{local}
Let $p,q,\mu$ be subject to \eqref{pq}, and assume $z_0 = (u_0,d_0)\in X_{\gamma,\mu}$, which means that $u_0,d_0\in B^{2\mu-2/p}_{qp}(\Omega)^n$ satisfy the compatibility conditions
$$ {\rm div}\, u_0= 0 \mbox{ in } \Omega,\quad u_0=\partial_\nu d_0=0 \mbox{ on } \partial\Omega.$$
Then for some $a=a(z_0)>0$, there is a unique solution
$$z\in H^1_{p,\mu}(J,X_0)\cap L_{p,\mu}(J;X_1),\quad J=[0,a],$$
of \eqref{eq:LCD} on $J$. Moreover,
$$z\in C([0,a];X_{\gamma,\mu})\cap C((0,a];X_\gamma),$$
i.e. the solution regularizes instantly in time. It depends continuously on $z_0$ and exists on a maximal time interval $J(z_0) = [0,t^+(z_0))$. Therefore problem
\eqref{eq:LCD}, i.e.\ \eqref{eq:qlproblem}, generates a local semi-flow in its natural
state space $X_{\gamma,\mu}$.
\end{theorem}

\begin{remark} 
Assuming that $2/p+n/q<1$,  for $\varepsilon > 0$ we may choose $\mu$ subject to~\eqref{pq} 
such that
\begin{align*}
 H_q^{1 + \frac{n}{q} + \varepsilon}(\Omega)^n \hookrightarrow B^{2\mu-2/p}_{qp}(\Omega)^n \hookrightarrow H_q^{1 + \frac{n}{q} - \varepsilon}(\Omega)^n
\end{align*}
due to Sobolev embeddings~\cite[Theorem 4.6.1]{Tri78}. %In particular, there exists $\alpha \in (0,1)$ with \begin{align*}
%C^{1+\alpha}(\Omega) \hookrightarrow B^{2\mu-2/p}_{qp}(\Omega)^n
%\end{align*}
Furthermore, we can choose $p,q$ large with %shows that %$u_0,d_0\in C^{1+\alpha}(\Omega)$ for some $\alpha>0$.
\begin{align*}
C^{1+\varepsilon}(\Omega)^n \hookrightarrow B^{2\mu-2/p}_{qp}(\Omega)^n. 
\end{align*}

Employing different time weights for $u$ and $d$, an inspection of the above proofs shows that 
the assertion of the above theorem remains true provided 
$u_0\in C^\alpha(\Omega)$.
\end{remark}

The following result tells that the condition \eqref{deq1} is preserved by
\eqref{eq:LCD}.

\begin{proposition}\label{lengthdpreserved}
Suppose that $\mu, p, q$ are satisfying \eqref{pq} and let  $z_0 = (u_0, d_0)\in X_{\gamma,\mu}$
with $|d_0| \equiv 1$, $a>0$. Let
\begin{align*}
  z\in H^1_{p,\mu}(J;X_0)\cap L_{p,\mu}(J;X_1)
\end{align*}
be a solution of \eqref{eq:LCD} on the interval $J=[0,a]$. Then $|d(t)| \equiv 1$ holds
for all $t \in [0, a]$.
\end{proposition}

\begin{proof}
  Setting $\varphi= |d|^2-1$  the elementary identities,
\begin{align*}
 \partial_t |d|^2 =   2d \cdot \partial_td, \quad \Delta|d|^2 = 2 \Delta d\cdot d + 2|\nabla d|^2, \quad \nabla|d|^2 = 2d\cdot\nabla d, \end{align*}
and multiplication with $d$ of the second line in \eqref{eq:LCD} yields the problem
\begin{align*}
  \left\{
  \begin{array}{rllll}
 \partial_t\varphi  + u\cdot\nabla\varphi &\!=\!& \Delta \varphi + 2|\nabla{d}|^2\varphi  && \text{in } \Omega\\
 \partial_\nu \varphi&\!=\!&0  &&\text{on } \partial\Omega,\\
 \varphi(0)&\!=\!&0 &&\text{in } \Omega,
 \end{array}\right.
\end{align*}
provided $|d_0|\equiv1$. Uniqueness of this parabolic convection-reaction diffusion
equations  yields $\varphi\equiv0$, i.e.\ $|d|\equiv1$.
\end{proof}

%\section{Regularity of Solutions}\label{regularityofsolutions}
As the nonlinearities $A$ and $F$ are real analytic we may employ Angenent's method (Proposition~\ref{qlregularity}) to obtain further regularity of the solutions of \eqref{eq:LCD}.

\begin{proposition}\label{regularity}
Suppose that $\mu, p, q$ satisfy~\eqref{pq}, $z_0\in X_{\gamma,\mu}$, and $a>0$ and let
\begin{align*}
  z\in H^1_{p,\mu}(J;X_0)\cap L_{p,\mu}(J;X_1)
\end{align*}
be a solution of \eqref{eq:LCD} on the interval $J=[0,a]$.
Then for each $k\in\N$,
\begin{align*}t^k [\frac{d}{dt}]^k z \in  H^1_{p,\mu}(J;X_0)\cap L_{p,\mu}(J;X_1).\end{align*}
Moreover, $z\in C^\omega((0,a); X_1)$.
\end{proposition}

We will employ Proposition \ref{regularity} in the following to justify the regularity of
time derivatives of the energy functional.

\begin{remark}
  \label{rkm:re}
Employing scaling techniques jointly in time and space, it is possible to show via
maximal regularity and the implicit function theorem that $u,\pi,d$ are
real analytic in $(0,t^+(z_0))\times \Omega$; see \cite[Section 5]{Pru03} for parabolic problems, and specifically for a Navier-Stokes problem \cite{PS11}.
As we will not use this result below we omit the details, here.
\end{remark}

\section{Stability and Convergence to Equilibria} \label{stabilityofequilibria}

We consider the set $\cE_0=\{0\}\times \R^n$, which are obviously equilibria of~\eqref{eq:LCD}. This set forms a $n$-dimensional subspace of $X_1$, hence a $C^1$-manifold with tangent 
space $\{0\}\times \R^n$ at each point $(0,d_*)\in \cE_0$ .
The linearization of \eqref{eq:LCD} at $z_*\in\cE_0$ is given by the linear evolution equation
$$
\dot{z} + A_*z = f,\quad  z(0)=z_0,
$$
in $X_0$, where
$$
A_* = {\rm diag}(\cA_q, \cD_q),\quad D(A_*)=X_1.
$$
As $\Omega$ is bounded, the spectrum $\sigma(\cA_q)$ consists only of positive eigenvalues
and $0\not\in \sigma(\cA_q)$. On the other hand, $\cD_q$ has $0$ as an eigenvalue, which is semi-simple and the remaining part of $\sigma(\cD_q)$ consist only of positive eigenvalues.
Thus $\sigma(A_*)\setminus\{0\} \subset [\delta,\infty)$ for some $\delta>0$ and
the kernel of $A_*$ is given by
$$
N(A_*)=\{0\}\times \R^n,
$$
which equals the tangent space. In Remark 2.2 in \cite{psz09} it is shown that all equilibria close to $z_*$ are contained in a manifold $\mathcal{M}$ of dimension $n = \operatorname{dim}(N(A_*))$. Since the dimension of $\mathcal{E}_0$ is also $n$, there exists an open set $V\subset X_1$ with $\mathcal{M} \cap V =  \mathcal{E} \cap V = \cE_0 \cap V$; i.e. $\mathcal{E} \cap V$ contains no other equilibrium. As a result we see that the equilibrium is normally stable.

%Then according to Remark 2.2 in \cite{psz09}, there exists an open set $\cE_0\subset V\subset X_1$ which contains no other equilibrium. In this remark it is in particular shown that all equilibria close to $z_*$ are contained in a manifold $\mathcal{M}$, which is of dimension $n = \operatorname{dim}(N(A_*))$. 
%As a result we see that the equilibrium is normally stable.

Now we are in position to apply Proposition \ref{thm:global1} to conclude the following
stability result for the equilibria of   \eqref{eq:LCD}.

\begin{theorem}\label{stability}
Let $p, q$ satisfy the first inequality in~\eqref{pq}. Then each equilibrium
$z_*\in \{0\}\times \R^n$ is stable in $X_\gamma$, i.e. there exists $\epsilon>0$ such that a
solution $z(t)$ of~\eqref{eq:LCD} with initial value $z_0\in X_\gamma$, $|z_0-z_*|_{X_\gamma}\leq \epsilon$, exists globally and converges exponentially to some
$z_\infty\in \{0\}\times \R^n$ in $X_\gamma$, as $t\to\infty$.
\end{theorem}

%\section{The Energy and Convergence of Solutions to Equilibria} \label{energyandconvergence}

We next consider the energy of the system given by
\begin{equation}\label{energy}
\sE = \frac{1}{2}\int_\Omega[|u|^2+|\nabla d|^2]dx =\sE_{kin}+\sE_{pot}.
\end{equation}
Let $z = H^1_{p,\mu}(J; X_0) \cap L_{p,\mu}(J; X_1)$
be a solution of~\eqref{eq:LCD}--\eqref{deq1} on $J = [0,a]$; according to Proposition 4.4 it belongs to $C^1((0,a);X_1)$. 
Using sum convention we have by  an integration by parts
\begin{align*}
\frac{d}{dt}\sE_{kin}(t) &=\int_\Omega \partial_t u \cdot u dx \\
&= \int_\Omega [-(u\cdot\nabla) u -\nabla\pi+ \Delta u- {\rm div}([\nabla d]^{\sf T}\nabla d])\cdot u dx\\
&= -\int_\Omega|\nabla u|^2dx + \int_\Omega \partial_k d_j\partial_id_j \partial_k u_i dx,\\
\end{align*}
as ${\rm div}\, u =0$ in $\Omega$ and $u=0$ on $\partial\Omega$. On the other hand,
we have by another integration by parts
\begin{align*}
\int_\Omega |\Delta d + |\nabla d|^2d|^2dx&= \int_\Omega[\Delta d +
|\nabla d|^2d][\partial_t d +(u\cdot\nabla) d]dx\\
&= -\int_\Omega [\partial_t \nabla d:\nabla d - |\nabla d|^2\partial_t|d|^2/2]dx\\
& \quad + \int_\Omega[ (u\cdot \nabla) d\cdot \Delta d + |\nabla d|^2 (u\cdot\nabla)|d|^2/2]dx\\
&= -\frac{d}{dt}\sE_{pot}(t) - \int_\Omega \partial_k(u_i\partial_id_j)\partial_k d_j dx\\
&=  -\frac{d}{dt} \sE_{pot}(t) - \int_\Omega \partial_ku_i\partial_id_j\partial_k d_j dx,
\end{align*}
by $|d|\equiv 1$ and % see the proof of Theorem~\ref{local},
%by
the Neumann boundary condition for $d$.
Combining these equations, we obtain the energy identity
\begin{equation}\label{enid}
\frac{d}{dt} \sE(t) = -\int_\Omega [|\nabla u|^2 +|\Delta d+|\nabla d|^2d|^2]dx.
\end{equation}
Therefore $\sE(t)$ is non-increasing along solutions. But $\sE$ is also a strict Ljapunov functional, i.e.\ strictly decreasing along non-constant solutions. In fact, if $d\sE(t)/dt=0$ at some time instant, then by the energy equality we have $\nabla u=0$ and $\Delta d +|\nabla d|^2 d=0$ in $\Omega$. Therefore $u=0$ by the no-slip condition on $\partial\Omega$, and $d$ satisfies the nonlinear eigenvalue problem
\begin{align}
  \left\{
  \label{nlevp}
  \begin{array}{rllll}
    \Delta d +|\nabla d|^2 d &\!=\!&0 &&\text{in } \Omega,\\
    |d|^2&\!=\!&1 &&\text{in } \Omega,\\
    \partial_\nu d &\!=\!&0 &&\text{on } \partial\Omega.
  \end{array}\right.
\end{align}
But, as the lemma below shows, this implies $\nabla d=0$ in $\Omega$, hence $d=d_*$ is constant and $z_*:=(0,d_*)$, $|d_*|=1$ is an equilibrium of the problem.%, $z_*\in\cE$.

\begin{lemma} \label{lma:dconst} Suppose that $d\in H^2_2(\Omega;\R^n)$ satisfies \eqref{nlevp}.
Then $d$ is constant in $\Omega$.
\end{lemma}

\begin{proof}
The idea is to reduce inductively the dimension $N=n$ of the vector $d$. This can be achieved by introducing polar coordinates according to
$$ d_1=c_1\cos\theta,\; d_2=c_1\sin\theta,\; d_j = c_{j-1},\quad j\geq3.$$
Simple computations yield
$$ 1=|d|^2 = |c|^2, \quad |\nabla d|^2 = |\nabla c|^2 + c_1^2|\nabla\theta|^2, $$
and
$$\Delta c_j +[|\nabla c|^2+c_1^2|\nabla\theta|^2]c_j=0 \quad \mbox{ in } \Omega,$$
as well as $\partial_\nu c_j=0$ on $\partial\Omega$ for $j=2,\ldots,n-1$.
Moreover, by an easy calculations we further obtain
$$-\Delta c_1 +c_1|\nabla \theta|^2= [|\nabla c|^2+c_1^2|\nabla\theta|^2]c_1\quad \mbox{ in } \Omega,$$
and
$$c_1\Delta \theta + 2\nabla c_1\cdot\nabla\theta=0\quad \mbox{ in } \Omega,$$
as well as
$$ \partial_\nu c_1= c_1\partial_\nu\theta=0\quad \mbox{ on } \partial\Omega.$$
Multiplying the former equation by $c_1\theta$ and integrating over $\Omega$ we deduce
\begin{align*}
0&=\int_\Omega [c_1\Delta\theta + 2 \nabla c_1\cdot\nabla \theta)]c_1\theta dx
= \int_\Omega {\rm div}[c_1^2\nabla\theta]\theta dx = -\int_\Omega c_1^2 |\nabla\theta|^2dx,
\end{align*}
hence $c_1\nabla\theta=0$. This implies that $c$ satisfies the problem \eqref{nlevp} where the vector $c$ has dimension $N-1$. Inductively, we arrive at dimension $N=1$ and if $d$ is a solution of \eqref{nlevp} with dimension 1, then $d=1$ or $d=-1$ by connectedness of $\Omega$.
\end{proof}

Note that the side condition $|d|\equiv 1$ is important at this point.
Summarizing we proved the following result.

\begin{proposition}
The energy functional $\sE$ defined on $X_\gamma$ is a strict Ljapunov function for system \eqref{eq:LCD}--\eqref{deq1}. The equilibria of this system are given by the set
$$\cE=\{ z_*=(u_*,d_*): \, u_*=0,\, d_*\in \R^n,\, |d_*|=1\},$$
which forms a manifold of dimension $n-1$.
The corresponding pressures $p_*$ are constant as well.
\end{proposition}

Suppose finally that  $z$ is a solution of \eqref{eq:LCD}--\eqref{deq1} which is
eventually bounded in $X_{\gamma}$ on its maximal interval of existence. Then,  by
Proposition~\ref{thm:global1} this solution is global and $z([\delta,\infty))\subset X_\gamma$
is relatively compact. Therefore its limit set
\begin{align*}
  \omega(z_0) =  \{ v \in X_\gamma\colon \exists t_n \uparrow \infty \text{ s.t. } z(t_n; z_0) \to v \text{ in } X_\gamma\}
\end{align*}
is nonempty. As $\sE$ is a strict Ljapunov functional for~\eqref{eq:LCD}--\eqref{deq1}, we obtain $\operatorname{dist}(z(t, z_0), \omega(z_0)) \to 0$ in $X_\gamma$ for $t\to \infty$ and $\omega(z_0)\subset \cE\subset X_1$. Now Theorem~\ref{stability} applies and we may conclude that $z(t)\to z_\infty\in\cE$ in $X_\gamma$ as $t\to\infty$. In summary we proved the following result. 

\begin{theorem}\label{globalconv}
Let $\mu, p, q$ satisfy~\eqref{pq}. Let $z_0 = (u_0, d_0)\in X_{\gamma,\mu}$ with $|d_0|\equiv 1$ and suppose that the solution $z(t)$ of \eqref{eq:LCD} is eventually bounded in $X_{\gamma}$ on its maximal interval of existence, i.e.
  \begin{align*}
    z \in BC([\tau, t^+(z_0)); X_{\gamma})
  \end{align*}
  for some $\tau \in (0,t^+(z_0))$. Then $z(t)$ exists globally and converges to an
equilibrium $z_\infty\in \cE$ in $X_\gamma$, as $t\to\infty$. The converse is also true.
\end{theorem}

\bibliographystyle{amsplain}
\bibliography{literature}

\end{document}